\documentclass{article}
\usepackage{amsfonts,amsmath,amssymb,amsthm}
\usepackage[all,2cell]{xy}\UseAllTwocells\SilentMatrices

\usepackage{color}
\usepackage[utf8]{inputenc}
\usepackage{titlesec}
\usepackage{hyperref}
\hypersetup{
    colorlinks,
    citecolor=black,
    filecolor=black,
    linkcolor=black,
    urlcolor=black
} 

\usepackage[margin=1.5in]{geometry}
\usepackage{setspace}
\usepackage{comment}
\usepackage{tabularx}
\usepackage{comment}

\label{guassthm} 

\title{Group actions, and D equivalences of categories of coherent sheaves of symplectic resolutions}
\author{D. Boger}
\date{}

	\newcommand{\bs}{\bigskip}

	\definecolor{red}{rgb}{1,0,0}

		\newcommand{\D}{\mathcal{D}}

	\newcommand{\subsubsubsection}[1]{\vspace{0.5cm}
 \textit{\textbf{#1}} \vspace{0.2cm}}

	\newcommand{\g}{\mathfrak{g}}

	\newcommand{\R}{\mathbb{R}}
	\newcommand{\CC}{\mathbb{C}}
	\newcommand{\C}{\mathcal{C}}
	\newcommand{\G}{\mathcal{G}}
	\newcommand{\A}{\mathcal{A}}

\begin{document}
\onehalfspacing
\maketitle
			
\newcommand{\Dl}{\mathcal{D_{\lambda}}}
\newcommand{\Aop}{\A_F^-}
\newcommand{\Dle}{\mathcal{D}_{\lambda,e}}
\newcommand{\ul}{u(\g)_{\lambda}}
\newcommand{\Al}{\Gamma(\Dl(G/P))}

\newtheorem{thm}{Theorem}
\newtheorem{lemma}{Lemma}
\newtheorem{claim}{Claim}
\newtheorem{defi}{Definition}
\newtheorem{cor}{Corollary}
\newtheorem*{example}{Example}
\newtheorem{remark}{Remark}
\newtheorem{observation}{Observation}

\begin{abstract}

Let G be a reductive group over an algebraically closed field k of characteristic $p>>0$. Let $\g$ be its Lie algebra. Let $\mathcal{B}$ be the variety of Borel subalgebras in $\g$.  
There are two questions that motivates this work. One is the question of constructing an action of a group on the category $\C:=D^b(Coh(X))$ - the derived category of coherent sheaves of symplectic resolution X. The second question is understanding equivalence functors between these categories for different symplectic resolution of the same symplectic singularity. In this paper, we answer the first question by constructing such an action for the case $X=T^*G/P$, P parabolic subgroup. This extends the affine braid group action on category of coherent sheaves, in the case of the well known springer resolution $T^*\mathcal{B}\rightarrow \mathcal{N}$. We construct the group action by constructing a local system of categories on a topological space called $V^0_{\CC}$, with value - the category $\C$ hence obtaining an action of $\pi_1 V^0_{\CC}$. We hope to prove a generalization of this construction is well defined for arbitrary symplectic resolution. In \cite{B1} we further explain how a refinement of this local system construction, gives an answer to the second question, showing that these equivalence functors, are parametrized by homotopy classes of maps between certain points in the base space $V^0_{\CC}$. We also lift the result to the case that the field k is of characteristic zero.

\end{abstract}

\tableofcontents
\bs

I'm very thankful to Prof. Bezrukavnikov, for his great help and support. 

\section{Notations and Basic Facts}

The symbols below have the following meaning in this paper: 

Let G be a reductive group over an algebraically closed field k, of characteristic $p>>0$. 

Fix a maximal torus T. Let $\Lambda:=Hom(T,\mathbb{G}_m)$ be the integral weight lattice. Assume all parabolic subgroups that we will mention, contain T. Let $L\supset T$ be a fixed Levi subgroup. Let  $\Lambda_L:=Hom(L,\mathbb{G}_m)=Hom(L/[L,L],\mathbb{G}_m)\subset \Lambda$. (The canonical embedding is given by restriction from L to T). Explicitly $\Lambda_L=\lambda\in \Lambda| <\lambda,\alpha^{\vee}>=0$ for all roots $\alpha$ with weight space $\g_{\alpha}\subset Lie(L)$ 
\bs

\begin{remark}
We summarize here the notations we will be using in the text, however, a more detailed explanation of them appears in the background section.
\end{remark}

\begin{defi}
		Let W denote the Weyl group. $\lambda\in \Lambda$ is called \emph{p-regular} (We will also call it regular, when there is no risk of confusion between char 0 and p), if the stabilizer in W, of $\lambda+p\Lambda\in \Lambda/p\Lambda$is trivial, (Under the action $w\bullet \lambda:=pw((\lambda+\rho)/\rho)-\rho$). Equivalently, $<\lambda+\rho,\alpha^{\vee}>\neq 0$ mod p for all coroots $\alpha$. 
\end{defi}
\bs

Most of of the weights in $(\Lambda_L\otimes \R)$ are p-regular as elements of $\Lambda$.
\bs

Let $P\subset G$ be a parabolic subgroup, with Levi L. 
\begin{defi}
Let $\lambda\in \Lambda_L$. Let $\mathcal{D}_{\lambda}(G/P)$ be the sheaf of $\lambda$ twisted crystalline differential operators on G/P.  (We will denote it as $\mathcal{D}_{\lambda}$ when there is no ambiguity). Let $D^b(\mathcal{D}_{\lambda}-mod)$ be the twisted D modules.
\end{defi}
\begin{defi}
Let $A_{\lambda}:=\Gamma(\mathcal{D}_{\lambda}(G/P))$ be the global sections of the sheaf $\mathcal{D}_{\lambda}(G/P)$. Let $A_{\lambda}-mod$ be the category of finitely generated $A_{\lambda}$ modules. 
\end{defi}

\begin{observation}
By derived localization theorem in characteristic p, if $\lambda$ is p regular, there is an equivalence of categories $\Gamma:D^b(\mathcal{D}_{\lambda}-mod)\rightarrow D^b(A_{\lambda}-mod)$. \cite{BMR} (\cite{Be1} for the classical BB localization, which this theorem is a derived analog of)
\end{observation}

\begin{defi}
Working over k, we can think of $\mathcal{D}(G/P)_{\lambda}$ modules as coherent sheaves on $T^*G/P$ with extra data. Let F be the Forgetful functor $F:\mathcal{D}_{\lambda}-mod\rightarrow Coh(T^*G/P)$.  
Let $Coh(T^*G/P)_0$ be the category of coherent sheaves, supported on the formal neighborhood of the zero section of $T^*G/P$. $G/P\subset T^*G/P$

Let $\mathcal{D}_{\lambda}-mod_0$ be $\lambda$ twisted D-modules $\mathcal{F}$, s.t the support of $F(\mathcal{F})$ is on the formal neighborhood of the zero section of $T^*G/P$. 

Let $A_{\lambda}^{\hat{0}}-mod$ be the category of finitely generated $A_{\lambda}$ modules, with generalized p center, zero.  

(Another notation that we will use for this category is $A_{\lambda}-mod_0$)
\end{defi}

\begin{observation}
The derived localization theorem induces an equivalence on the full subcategories $D^b(\mathcal{D}_{\lambda}-mod_0)\simeq D^b(A_{\lambda}^{\hat{0}}-mod)$. (\cite{BMR2}). 
\end{observation}

\begin{example}
Let $U(\g)$ be the enveloping algebra.
Let P be a borel B subgroup. The category $A^{\hat{0}}_{\lambda}-mod$ is $U(\g)$ modules, whose Harich Chandra center is $\lambda$ and with whose generalized p-center is zero.
\end{example}

\begin{observation}
There is a surjective morphism $U(\g)_{\lambda}\rightarrow A_{\lambda}$, and $U(\g)^{\hat{0}}_{\lambda}\rightarrow A^{\hat{0}}_{\lambda}$. 
\end{observation}

Let $\mathcal{N}\subset \g$ be the nilpotent cone. Let $\pi:T^*G/P\rightarrow \mathcal{N}$ be the moment map. (Also called parabolic springer map). Let $\mathcal{N}_L$ be its image. 
\bs

Remark: In the above notations and in the rest of the paper - we made a choice to work with coherent sheaves and D modules, supported on a formal neighborhood of $\pi^{-1}(0)$,  the zero section of $T^*G/P$. This can be generalized to support on formal neighborhoods of other parabolic springer fibers as well.  

\bs

Notations: In this paper, we use the letters Q,P to denote associate Parabolic subgroups with the same Levi subgroup L.

\section{Motivation and statement}

Two questions form the motivation for this work: For G/k , $k=\bar{k}, char(k)=p>>0$ as above,  $\g:=Lie(G)$, B a Borel subgroup, $\mathcal{N}$-the nilpotent cone in $\g$. The springer resolution $\pi:T^*G/B\rightarrow \mathcal{N}$ is a symplectic resolution. Its derived category $D^b(Coh_0(T^*{G/B}))$ (0 stands for the support condition) carries a known action of the affine braid group $Br_{aff}$ (\cite{Ric}\cite{BR}\cite{BM}\cite{BMR}\cite{Be1}. See also \cite{AB} for different realization of the action). One would like to find analogous actions for other symplectic resolutions. In this work we build a construction for the case $X=T^*G/P$ (a symplectic resolution of a finite cover of $\mathcal{N}_L$). It's a construction of a local system of categories on a simple topological space $V^0_{\CC}$, with value $\C:=D^b(Coh_0(T^*G/P))$. This gives rise to an action of $\pi_1(V^0)_{\CC}$ on the category $\C$. In the case of P=B, this recovers the action of $Br_{aff}$ on the category. More precisely, $\pi_1(V^0_{\CC})=Br_{aff,pure}$ - the pure affine braid group. However in this case there is additional symmetry in the construction, that allows the action to extend to action of $Br_{aff}$. In \cite{B1}, we also explain how to lift this action to case of G over a char 0 field and generalize to the category of coherent sheaves without the support condition.

A central role in the construction is played by the topic of quantizations in characteristic p. The base of the construction is subset $V^0_{\CC}$ of the universal parameter space of quantizations of symplectic resolutions. The quantizations of a symplectic resolution X, which are parametrized by points in $V^0_{\CC}$, give rise to various t structures on the category $D^b(Coh(X))$. The local system that we construct enables a description of the variation of the t structures along the parameter space $V^0_{\CC}$. (\cite{B2})

\bs

A second question that motivated the work is Kawamata's 'K equivalence implies D equivalence conjecture': two smooth projective varieties X,X' over algebraically closed field, which are birationally equivalent, are called K equivalent, if there is a birational correspondence $X\leftarrow Z \rightarrow X'$, for Z smooth projective variety, such that the pullbacks to Z of the canonical divisors are linear equivalent. The conjecture is that K equivalence implies the equivalence of the bounded derived categories of coherent sheaves.  One caveat of this conjecture, is that one do not expect to get a canonical equivalence. Indeed, a special case where K equivalence is satisfied - is two symplectic resolutions X,X' of the same symplectic singularity Y. In this case, the D equivalence was proved by Kaledin. The key to his construction of equivalence functor , is a choice of a tilting generator of $D^bCoh(X)$, which is a noncanonicaly defined object. Hence, the second question is understanding the family of equivalences between the derived categories. In \cite{B1}, we explain that a refinement of the above local system, gives a parametrization for natural equivalence functors between $D^b(Coh(X))\rightarrow D^b(Coh(X'))$ by homotopy classes of maps between certain points in the base space $V^0_{\CC}$.

\section{Background}

			\subsection*{D modules and Coherent sheaves in characteristic p}
			
					Let k be a perfect field of characteristic $p>>0$. Let X be a smooth variety over k. 
					
					Let $\mathcal{D}_X$ be the sheaf of crystalline differential operators on X. We recall the relations of $\mathcal{D}_X$ with the structure sheaf of the cotangent space $O_{T^*X}$ and recall the relations between $\mathcal{D}_X$ modules and coherent sheaves on $T^*X$. Two main ideas are 1. $\mathcal{D}_X$ is a quantization of the sheaf $O_{T^*X}$. 2. $\mathcal{D}_X$ is a quasi coherent sheaf of algebras on $T^*X^{(1)}$ because the center of $\mathcal{D}_X$ is canonically isomorphic to the sheaf $O_{T^*X^{(1)}}$. The superscript (1) stands for Frobenius twist. For a variety over a perfect field k of char p, which is defined over $F_p$, the variety and its Frobenius twist are isomorphic as k schemes, hence we usually omit the superscript from the notation and identify $T^*X^{(1)}\simeq T^*X$.
					\bs
					
					In this subsection we focus on observations and examples. For more see \cite{BM}\cite{BMR}.

					\begin{defi}
					Let $\mathcal{D}_X$ be the sheaf of crystalline differential operators on X. It assigns for an affine open subset $U\subset X$, an algebra $D(U)$.
					
					$D(U)$ is defined to be generated by $O(U)$ and $Vect(U)$ (Vector fields). The relations are: 
					
					Let $\xi_1,\xi_2\in Vect(U)$, $\xi_1\xi_2-\xi_2\xi_1=[\xi_1,\xi_2]\in Vect(U)$. 
					
					Let $\xi_1\in Vect(U), f\in O(U)$. $\xi f-f\xi=\xi(f)$. 
					
					$O(U)\subset D(U)$ is a subalgebra.
					
					\end{defi}
					
					Remark: if k was an algebraically closed field of characterstic zero, this is the definition by generators and relations of the ordinary sheaf of differential operators in characteristic zero. 
					\bs
					
					Twisted differential operators. Let $\mathcal{L}\in Pic(X)$ be a line bundle. Then the sheaf of $\mathcal{L}$ twisted differential operators  $\mathcal{D}_X^{\mathcal{L}}$ is well defined. 
					\bs
					
					Notation: Let $X:=G/B$ let $\mathcal{D}_{\lambda}:=\mathcal{D}^{\mathcal{L}}$, where $\mathcal{L}$ corresponds to $\lambda$ under the equivalence $Pic(G/B)\simeq \Lambda$. Similarly for $X:=G/P$ where $Pic(G/P)\simeq \Lambda_L$
				\bs
				
				Notation: Let $\mathcal{D}_X^{\mathcal{L}}-mod$ be the category of $\mathcal{L}$ twisted D modules on X. (For X:=G/P we also denote it $\mathcal{D}_{X,\lambda}-mod$ or $\mathcal{D}_{\lambda}-mod$ when there is no ambiguity)
					\begin{claim}
					It is still true, as in characteristic 0, that $\mathcal{D}_X$ is a quantization of $O_{T^*X}$. That is, there is a natural filtration on this sheaf and $gr(\mathcal{D}_X)\simeq O_{T^*X}$. 
					\end{claim}
					
					\begin{claim}
					In contrary to the characteristic zero situation, the an element in $D(U)$ is not uniqly identified by its action on $O_X(U)$. The morphism $\mathcal{D}_X\rightarrow End(O_X)$ isn't injective. The basic example to keep in mind, is where X is the affine line, and the differential operator is $\partial^p$. It's not zero as an element of $\mathcal{D}_X$, yet it acts by zero on $O_X$.  
					\end{claim}
							
					\begin{thm}
					The center of $\mathcal{D}_X$ is big. It's canonically isomorphic to the sheaf of rings $O_{(T^*X)^{(1)}}$
					\end{thm}
					
					\begin{thm}
							$\mathcal{D}_X$, and $\mathcal{D}_X^{\mathcal{L}}$ are both Azumaya algebras on $T^*X^{(1)}$.
					\end{thm}
				
						Given an Azumaya algebra $\mathcal{A}$ over a variety Y, we can consider the category of coherent $\mathcal{A}$ modules. 
						
						When an Azumaya algebra is split over Y. That is, $\mathcal{A}\sim O_Y$, then the category of coherent $\mathcal{A}$ modules is equivalent to Coh(Y). 
						
						Consider Y:=$T^*X$. 
						
						$\mathcal{D}_X^{\mathcal{L}}$ and $\mathcal{D}_X$ are Morita equivalent. There is a canonical equivalence between the category of $\mathcal{D}_X$ modules and the category of $\mathcal{D}_X^{\mathcal{L}}$ modules. 
						
					\begin{thm}	
						$\mathcal{A}:=\mathcal{D}_X$ doesn't split over $T^*X$, but it does splits on a formal neighborhood of the zero section of $T^*X$. 
										
					The splitting implies the equivalence:
					\begin{equation}
					\label{D-coh}
					\mathcal{D}_X^{\mathcal{L}}-mod_0\simeq Coh_0(T^*X)
					\end{equation}
					Where in both side we restrict to sheaves with support on the formal neighborhood of the zero section $X\subset T^*X$.
					\end{thm}
					
					In the special case that $X=T^*\mathcal{B}$ (Or $T^*G/P$) this can be generalized.
					\begin{thm}
					
						Consider $X=T^*\mathcal{B}$  
						
						Consider the springer map $\pi:T^*G/B\rightarrow \mathcal{N}$. 
									$\mathcal{D}_{G/B}^{\lambda}$ splits on the formal neighborhood of every fiber of the springer resolution $\pi$. Hence obtain an equivalence 
					\begin{equation}
					\mathcal{D}_{G/B}^{\mathcal{L}}-mod_{e}\simeq Coh_e(T^*(G/B))
					\end{equation}
					
					Where the subscript stands for a support condition on a formal neighborhood of $\pi^{-1}(e)$ for $e\in \mathcal{N}$. 
					\end{thm}
					
					We won't use this generality though in this work.		
			
			\subsection*{Localization theorem in characteristic p}
			
			Let $\mathfrak{h}$ be the universal Cartan subalgebra of $\g$. Let $B\subset G$ be any borel subgroup of G, and $\mathfrak{b}$ be its lie algbera, then there is a canonical isomorphism $\mathfrak{h}\simeq \mathfrak{b}/[\mathfrak{b},\mathfrak{b}]$. Let $T\subset B$ be a maximal torus contained in B. Let $\mathfrak{t}$ be its lie algebra.
			
			Let W be the Weyl group. W acts on the weight lattice $\Lambda$ by the dot action. 
						
			Define $\xi$ to be the center of $U(\g)$. The center consists of two parts. Let $\xi_{HC}$ be the Harich-Chandra center. $\xi_{HC}:=U(\g)^G\simeq S(\mathfrak{t})^{(W,\bullet)}$. $S(\mathfrak{t})$ is the symmetric algebra of $\mathfrak{t}$, and we take the $W$ invariants for the $\bullet$ action.
			The second part is the Frobenius center $\xi_{Fr}$, also called p-center.   $\xi_{Fr}$ is an algebra generated by the elements of the form $x^p-x^{[p]}$ $x\in \g$, where $x\mapsto x^{[p]}$ is the restricted power map, which is characterized by $ad(x^{[p]})=ad(x)^p$. Maximal ideals of $\xi_{Fr}$ are in bijection with points of $\g^*\simeq \g$. $\xi_{Fr}\simeq S(\g^{(1)})$. That is functions on $\g^{*(1)}$, (the frobenius twist of $\g^*$).
				
				$\xi$ is build from these two parts. When $p>>0$, $\xi_{HC}\otimes_{\xi_{Fr}\cap \xi_{HC}}\xi_{Fr}\simeq \xi$. In other words, to give a central character of $U(\g)$, we need to give a pair of elements $(\lambda,e)\in \mathfrak{t^*}\times \g^{*(1)}$, which are compatible. 
				
				In particular lets restrict to the case that e is nilpotent and $\lambda$ integral weight.

				\bs
				
				Let $\lambda\in \Lambda$. There is a natural map $\Lambda/p\Lambda\rightarrow \mathfrak{h}^*/W$ $\lambda\mapsto d\lambda$ mod W. Hence $\lambda$ defines a maximal ideal of $\xi_{HC}$. We define $U(\g)_{\lambda}:=U\otimes_{\xi_{HC}}k$. 
				
				Let $U(\g)_{\lambda}-mod$ be the category of finitely generated $U(\g)_{\lambda}$ modules.
				
				$\lambda\in \Lambda$ is called \emph{p-regular}, if the stabilizer in W of $\lambda+p\Lambda\in \Lambda/p\Lambda$ is trivial. 
				
				\bs
				
				Given a compatible pair $\lambda\in \Lambda,e\in \g^*$. We consider $U_{\lambda}^{\hat{e}}$-mod. The category of finitely generated $U_{\lambda}$ modules with generalized p character $e$. That is, $U_{\lambda}$ modules for which there exists a power of the maximal ideal e in $\xi_{Fr}$ which kills them. We also denote this category by $U_{\lambda}-mod_{e}$.

				\subsubsection*{Derived localization in char p}
								
						Let $\lambda\in \Lambda$.
						
						\begin{thm}
						There is a natural isomorphism $\Gamma(D_{{G/B},\lambda})\simeq U(\g)_{\lambda}$. 
						\end{thm}
								
						Recall that the notation $D^b(\mathcal{D}_{\lambda}-mod)$ stands for the category of $\lambda$ twisted D modules. 
						
						\begin{thm}
						Let $\lambda\in \Lambda$ be p-regular, then derived global sections functor $\Gamma:D^b(\mathcal{D}_{G/B\lambda}-mod)\rightarrow D^b(U_{\lambda}-mod)$ is an equivalence of categories. \cite{BMR}
						\end{thm}
						
						\begin{thm}
						Let $\lambda,e$ be a compatible pair. The derived localization induces an equivalences on the subcategories $D^b(\mathcal{D}_{\lambda}-mod_e)\simeq D^b(u_{\lambda}-mod_e)$
						\end{thm}
				
						Note that combining the last theorem with the equivlanece in equation (\ref{D-coh}), the following equivalence is obtained: $D^b(Coh_0(T^*G/B))\simeq D^b(\mathcal{D}(G/B)_{\lambda}-mod_0)\simeq D^b(u(\g)_{\lambda}^{\hat{0}}-mod)$.

			\begin{thm}
			
			Similarly, Let $X:=G/P$, P a parabolic. Let $A_{\lambda}:=\Gamma(G/P,\mathcal{D}_{G/P,\lambda})$. Let $\lambda\in \Lambda_L$ be p regular. then the derived global sections functor $\Gamma:D^b(\mathcal{D}_{G/P,\lambda}-mod)\rightarrow D^b(A_{\lambda}-mod)$ is an equivalence of categories, \cite{BMR2} , and  $D^b(Coh_0(T^*G/P))\simeq D^b(\mathcal{D}(G/P)_{\lambda}-mod_0)\simeq D^b(A_{\lambda}^{\hat{0}}-mod)$. 
			\end{thm}
			
			\subsection*{The moment map} \cite{Gi3}
			
					Let G be a reductive over algebraically closed field k. Let $\g:=Lie(G)$. Let Z/k be a variety with an action of G. The varieties $T^*Z$ and $\g^*$ are poisson varieties. The G action on Z, gives rise to a G equivariant morphism of Poisson varieties $T^*Z\rightarrow \g^*$. We compose with the natural isomorphism $\g^*\simeq \g$ from the killing form, and refer to $\mu:T^*Z\rightarrow \g$ as the moment map.
					
					In this paper, we consider the moment map for the case Z=G/P, where P is a parabolic subgroup. The map $\mu:T^*G/P\rightarrow \g$. 
					
					In the case P=B a Borel, the image of this map is the nilpotent cone $\mathcal{N}$, and the map is a symplectic resolution-the springer resolution. More generally Let L be the Levi subgroup of P. The image of this map is the closure of a nilpotent orbit, which we denote $\mathcal{N}_L\subset \g$. (The notation $\mathcal{N}_L$ is the emphasize that it only depends on the Levi.) The map $T^*G/P\rightarrow \mathcal{N}_L$ is not quite a symplectic resolution but rather generically a finite cover.

			\subsection*{Symplectic resolutions}
			
			A symplectic singularity Y, is a normal variety, with an algebraic symplectic 2 form, $\Omega$, on the smooth locus $Y^{sm}$ of Y, s.t for some smooth projective resolution of Y, $\pi:X\rightarrow Y$ the pullback of $\Omega$ from the smooth locus extends to a 2 form on all of X, possibly degenerate. 
			
		 Beaville showed that if $\Omega$ extends to one smooth resolution, then it extends to any other resolution.\cite{Beau}
			
			A symplectic resolution X of a symplectic singularity Y, is a resolution as above, where the 2 form on X is a symplectic form. 
			
			When the symplectic singularity Y is affine, Y must be the affinization of X. $Y:=X^{aff}$ $X^{aff}:=spec \Gamma(X,O_X)$. Hence in such case it's enough to specify X, without specifying Y. 
			\bs
			
			Observation: A symplectic manifold, is in particular a Poisson variety.
			
			Observation: The canonical bundle of a symplectic resolution is trivial $K_X\simeq O_{X}$. 
			\bs
			
			Standard examples of symplectic resolutions include: 
			\begin{enumerate}
			
			\item The springer resolution $\mu:T^*G/B\rightarrow \mathcal{N}$, or more generally the restriction of this resolution to a Slodowy slice.  
			
			\textbf{Slodowy slice:}\cite{Sl} Let $e\in \mathcal{N}$ be a nilpotent element.
			
			Slodowy slice, $\Sigma_e\subset \mathcal{N}$ is a subvariety of $\mathcal{N}$, s.t the restriction of the springer resolution $\mu$ to $\mu^{-1}(\Sigma_e)$ is another symplectic resolution. In other words, the preimage of $\Sigma_e$ under the springer resolution is a smooth manifold, and the restriction of the symplectic 2 form on $T^*G/B$ to $\mu^{-1}\Sigma_e$ is non degenerate.
			
			To define $\Sigma_e$, use Jacobson Morozov theorem to get $sl_2$ triple $e,h,f$ that includes e. And define $\Sigma_e$ to be the intersection of e plus the centralizer of f in $\g$, with the nilpotent cone. One place where the Slodowy slice plays a role is, using the symplectic resolution given by restricting the springer resolution to the slodowy slice, one generalizes the known Braid group action on the category $D^bCoh_{\mu^{-1}(0)}(T^*G/B)$ (coherent sheaves with support on formal neighhbhorhod of the zero section $G/B\subset T^*G/B$), to an action on the bounded derived category of coherent sheaves with support on $\mu^{-1}(e)$ for any $e\in{\mathcal{N}}$. (\cite{BM},\cite{BMR})
			
			\item Quiver varieties,\cite{Nak}. These varieties have important applications in representation theory. For example, in Nakajima's geometric construction of the representations and crystals of Kac Moody algebras. 
			
			\item Symplectic resolutions of quotient singularities. \cite{Gi1}
			A basic example is symplectic resolution of Kleinian singularities. Let $\Gamma \subset SL_2(\CC)$ be a finite subgroup. Consider the quotient $\CC^2/\Gamma$. This is a poisson variety, where the poisson structure is induced from the canonical symplectic structure on $\CC^2$. The canonical minimal resultion of $\CC^2/\Gamma$ is a symplectic resolution. (\cite{BK3}, \cite{Kal3},\cite{Kal4})
			
			More generally, we can consider any finite dimensional symplectic vector space, $(W,\omega)$ and a symplectic subgroup $G \subset Sp(W,\omega)$. Consider the Poisson variety $W/G$. where the Poisson structure is induced from the symplectic form on W. Symplectic resolutions of such varieties (when exists), have been studied.

			\textbf{Remark: Slodowy slice and Kleinian singularities}
			McKay correspondence gives a correspondence between finite subgroups as above $\Gamma\subset SL_2(\CC)$, and types A,D,E Dynkin graphs. 
			Recall - Let $\Gamma\subset SL_2(\CC)$ be a finite subgroup. Let
			I:=Irreps($\Gamma$) parametrize the irreducible representations $L_i$, $i\in I$ of $\Gamma$. the corresponding Dynkin diagram is defined by letting $e_{ij}:=dimHom_{\Gamma}(V_i,V_j\otimes \CC^2)$, where $\CC^2$ is the standard $SL_2$ representation, restricted to $\Gamma$. (Observe that $e_{ij}=e_{ji}$)
			
			Given $\Gamma$, take $\g$ to be of the corresponding type. Let $e\in \g$ be a subregular nilpotent. Then there is an equivalence of Poisson varieties between the quotient singularity and the Slodowy slice:$\CC^2/\Gamma\simeq \Sigma_e$.
		
			\bs
			
			Another quotient singularity that has been studied is the following. Let $\g$ be a finite dimensional semi simple lie algebra. Assume $\g$ is of type A, B, or C. Let $\mathfrak{h}$ be its cartan. When taking the vector space W, to be the cotangent bundle of the cartan, the weyl group $\mathbb{W}\subset GL(\mathfrak{h})$, is embedded in the symplectic group $Sp(W)$. In this setup, the quotient $W/\mathbb{W}$ has a symplectic resolution. (In fact the condition that $\g$ was of type A,B,or C is also necessary). More generally one could take any coxeter group rather than $\mathbb{W}$. 
		
			Moreover, specializing the above setup to type A, a symplectic resolution of the quotient space is well known. This is the Hilbert Chow morphism. $\pi:Hilb^n(\CC^2)\rightarrow T^*\CC^n/S_n$. 
			
			\end{enumerate}

			\subsection*{Quantizations of symplectic resolutions}
			
			Examples for filtered quantizations of a structure sheaf of a (poisson) variety, include: $U(\g)$ as the filtered quantization for $Sym(\g)$, symplectic reflection algebras as quantizations of quotient singularities $V/\Gamma$ for $(V,w)$ symplectic finite dimensional vector space and $\Gamma\subset Sp(V)$ finite subgroup, and the sheaf of differential operators $\mathcal{D}(X)$ for a cotangent space $T^*X$. 
			\bs
			
			The problem of universal quantization of symplectic resolutions has been studied in \cite{KV}\cite{BK1}\cite{BK2}.
			
			Given a symplectic resolution $X\rightarrow Y$, The universal parameter space for quantizations of X is $H^2(X,\R)$ \cite{BK1}\cite{BK2}\cite{KV}. The starting point of the proof is the fact that formal locally there exists only one quantization. 
			
			\bs
			
			For every $\lambda\in H^2(X,\R)$ let $O(X)^{\lambda}$ be the corresponding quantization of the structure sheaf. Observe - For any symplectic resolution $Pic(X)\otimes \R\simeq H^2(X,\R)$. When $X=T^*G/P$ these quantizations for $\lambda\in Pic(X)$ come from the sheaf of $\lambda$ twisted differential operators on G/P. (Recall that $Pic(T^*G/P) \simeq Pic(G/P)$)
			\bs
			
			For different symplectic resolutions $X^{(i)},X^{(j)}$ of the same symplectic singularity, the birational isomorphism between $X^{(i)},X^{(j)}$  induces an isomorphism on the Picard groups $Pic:=Pic(X^{(i)})\simeq Pic(X^{(j)})$. Hence it makes sense to compare the quantizations $O_{X^{(i)}}^{\lambda}$ and $O_{X^{(j)}}^{\lambda}$ parametrized by the same $\lambda\in Pic$

		\subsection*{Derived Localization in characteristic p}

		Fix a symplectic resolution $X\rightarrow Y$. Let $V_{\R}:=H^2(X,\R)\simeq Pic(X)\otimes \R$. One can ask when does derived localization hold, when is the global sections functor from sheaves over $O^{\lambda}(X)$, to modules over the global sections $\Gamma:D^b(O^{\lambda}(X)-mod)\rightarrow D^b(\Gamma(O_{\lambda}(X))-mod)$ an equivalence. e.g - Let $X=T^*G/B$, then derived localization holds when $\lambda$ is regular. This fact is a derived characteristic p version of Beilinson Berenstein localization theorem\cite{BeiBer}.
		
		More generally, for any symplectic resolution, it's a conjecture that derived localization holds away from a discrete set of hyperplanes $H_i$. Let me call these hyperplanes the 'walls' in $V_{\R}$.

		Let $V_{\R}^0\subset V_{\R}$ be the complement of these walls.	
			
			\subsection*{D equivalence of symplectic resolutions}
			
			\begin{defi}
			two smooth projective varieties X,X' over algebraically closed field, which are birationally equivalent, are called K equivalent, if there is a birational correspondence $X\leftarrow Z \rightarrow X'$, for Z smooth projective variety, such that the pullbacks to Z of the canonical divisors are linear equivalent. 
			\end{defi}

			Kawamata conjectures that K equivalence implies an equivalence of the bounded derived categories. 
			
			One case in which this holds is for two symplectic resolutions. K equivalence holds, since the canonical bundles of symplectic resolutions are trivial. It's a proof of Kaledin that there is also an equivalence of the derived categories.
			\begin{thm}
					Given X,X' symplectic resolutions of a fixed symplectic singularity Y, there is an equivalence $D^b(X)\simeq D^b(X')$.
			\end{thm}
			
						\begin{defi}
							A tilting generator for a smooth variety X/k, is a locally free sheaf $\mathcal{F}$ on X, s.t 1. $Ext^i(\mathcal{F},\mathcal{F})=0$, for all $i>0$, 2. The functor $RHom(\mathcal{F},-):D^bCoh(X)\rightarrow D^b(A_{\mathcal{F}}-mod)$ is an equivalence of categories. $A_{\mathcal{F}}:=End(\mathcal{F})^{op}$
						   
						\end{defi}
						
			\begin{thm}
						A symplectic resolution has a tilting generator. 
			\end{thm}
				
			\begin{proof} 
				A sketch of the D equivalence proof - Let $\mathcal{E}$ be a tilting generator for X. (That's a choice). X and X' agree on an open subset whose complement is of codim $\geq 2$. $\mathcal{E}$ can be extended to a vector bundle $\mathcal{E}'$ on X', s.t $H^i(X',End(\mathcal{E}'))=0$ for $i>0$. (yet $\mathcal{E}'$ is not necessarily the tilting generator of X). $\mathcal{E},\mathcal{E}'$ agree on an open set whose complement is of codim $\geq 2$, and the algebras $R:=End(\mathcal{E})\simeq End(\mathcal{E}')=:R'$ are isomorphic. Hence can consider the functor $Hom(\mathcal{E'},-):D^b(X')\rightarrow D^b(R'^{op}-mod)\simeq D^b(X)$. That's the functor that is proved by Kaledin to be an equivalence. The proof uses a trick related to $D^bcoh(X)$ being a Calabi-Yau category. \cite{Gi2}

			\end{proof}			
					
			The caveat is that the equivalence constructed is highly non canonical, since a tilting generator isn't. The local system constructed below has a refinement that leads to a better understanding of family of natural D equivalences between this categories.
			
\subsection*{The classical action on Grothendick group of $D^b(Coh_0(T^*G/B))$}		
					For simplicity we work over the complex numbers in this section. 
					Let $\pi:T^*G/B\rightarrow \mathcal{N}$ be the springer resolution.	Let $e\in \mathcal{N}$, Let $\mathcal{B}_e\subset T^*G/B$ be the springer fiber over e. $\mathcal{B}_e=$ set of Borels $B\subset G$ s.t $e\in Lie(\mathfrak{n}_B)$. 
				
					There is a known action of (affine) braid group on $D^b(Coh_{e}(T^*G/B))$, (The subscript e stands for restricting the support to be on the formal neighborhood of $\pi^{-1}(e)$). One of the goals of the local system we build, with value the category $D^b(Coh_0(T^*G/P))$ (P parabolic), is to generalize this action.
					
					The action on $D^b(Coh_eT^*(G/B))$ is a categorification of an extension of a known action of the weyl group on the cohomology of the springer fiber $\mathcal{B}_e$. That is, at the level of Grothendieck group, $K^0D^b(Coh_{e}T^*G/B)$, there is an equivalence $K^0D^b(Coh_{e}T^*G/B)\simeq H_*(\mathcal{B}_e)$. The action induced on $H_*(\mathcal{B}_e)$ is dual to the classical action on the cohomology of springer fiber by the Weyl group. We briefly recall that action.

						\subsubsection*{Properties of the springer fiber} 
		
		Let $e\in \mathcal{N}$, $\mathcal{B}_e$= set of borels $B\subset G$ s.t $e\in Lie( \mathfrak{n}_B)$. e.g Let $G=SL_n=SL(V)$, V/k n dimensional vector space. Identifing Borel subgroups with complete flags in V, $\mathcal{B}_e$ is flags that e preserves. ($0\subset V_1\subset..\subset V_{n-1}\subset V, e.V_i\subset V_{i}$)
		\bs
				
		Properties of $\mathcal{B}_e$: $dim \mathcal{B}_e=1/2(dim \mathcal{B}-dim O_{e})$ where $O_e$ is the adjoint orbit of e in $\mathcal{N}$. (note that the codimension of nilpotent orbit is always even, so the formula makes sense). In addition $\mathcal{B}_e$ is always connected, in interesting cases irreducible, and often singular. 
		\bs
		
		Extreme cases for springer fibers are $e=0$, for which $\mathcal{B}_e=\mathcal{B}$. and regular nilpotent e, for which $\mathcal{B}_e=pt$ (indeed the regular locus of $\mathcal{N}$, is exactly where the birational morphism $T^*(G/B)\rightarrow \mathcal{N}$ is an isomorphism).
		Explicitly for $SL_n$, regular nilpotent e is (up to conjugation) a single Jordan block in some basis $v_1..v_n\in V$, and the single Borel in the fiber of the springer map over e, is the upper triangular matrices in this basis. The corresponding flag is that which is determined by this basis $<v_1>\subset <v_1,v_2>\subset,..$
		\bs
		
	Explicit examples:
	
	Example 1. Let $G=SL_3$, e be of Jordan type (2,1). Then $\mathcal{B}_e\simeq \mathbb{P}^1\cup_{pt}\mathbb{P}^1$. (two projective lines, glued at a point). Lets write e in a basis $v_1,v_2,v_3\in V$ as a matrix which has all zeros except for upper right corner. Then explicitly in terms of flags, the first $\mathbb{P}^1$ is flags of the form $0\subset <v_1>\subset V_2\subset V$, where $V_2$ is any 2 dimensional sub vector space between the fixed $<v_1>$ and V. The second $\mathbb{P}^1$ is flags of the form $0\subset V_1\subset <v_1,v_2>\subset V$ and again, the free component is $V_1$ which is one dimensional vector space between 0 and $<v_1,v_2>$. These sets of flags indeed have one point in common, and each is isomorphic to $\mathbb{P}^1$. dim $\mathcal{B}_e=1$
		
		Example 2. Let $G=Sl_4$, let e be of Jordan type (2,2). Then there are two irreducible components of $\mathcal{B}_e$: $\mathbb{P}^1\times \mathbb{P}^1$ and $\mathbb{P}(O\oplus O(-2))$ ($\mathbb{P}$ stands for projectivization). They are glued along a $\mathbb{P}^1$ that is embedded in $\mathbb{P}^1\times \mathbb{P}^1$ diagonally and as the zero section of the other component. dim $\mathcal{B}_e=2$

\subsubsection*{The use of Grothendieck resolution $\tilde{\g}\rightarrow \g$ in the construction of the action}

		The action at the level of Grothendieck groups was constructed using the Grothendieck resolution. In fact, the action at the level of categories, on $D^b(Coh_e(T^*G/B))$, using correspondences as in \cite{BR}, also used the Grothendieck resolution. There is an underlying reason for that. 
		
		Grothendieck resolution is obviously closely related to $T^*G/B\rightarrow \mathcal{N}$. At the level of k points $\tilde{\g}:=(x,B)$ s.t $x\in \g,  B\in \mathcal{B}, x\in Lie(B)$, hence at that level, restricting $\tilde{\g}\rightarrow \g$ to $\mathcal{N}$ at the basis, gives you preimage $T^*G/B$. (Remark:That's not true at the level of schemes. The preimage of $\mathcal{N}\subset \g$ is a non reduced scheme. Its associated reduced scheme is $T^*(G/B)$). (See \cite{MR}, \cite{Be2} for more)
		
		$\tilde{\g}\rightarrow \g$ is very useful because it is \emph{small} whereas the springer map $T^*G/B\rightarrow \mathcal{N}$ is only \emph{semi-small}. That's the reason that some constructions work better using $\tilde{\g}\rightarrow \g$, and then with more work one may make them work for $T^*G/B\rightarrow \mathcal{N}$.

		\subsubsection*{The Weyl group action on $H^*(\mathcal{B}_e)$}
				Let W be the Weyl group. For every nilpotent $e\in \mathcal{N}$ there is an action on W on the cohomology of the springer fiber $H^*(\mathcal{B}_e)$. In general, there is no action of W on the space $\mathcal{B}_e$.
				
				Considering the extreme cases. The case of regular e is not of interest since, $\mathcal{B}_e$ is a point. \textbf{Case e=0:} For e=0 one gets an interesting action on $H^*(\mathcal{B})$. In this case, the action coincide with an older known action that is described in algebraic terms: The cohomology ring of G/B is well understood. There is an isomorphism $H^*(\mathcal{B})\rightarrow sym(\mathfrak{h}^*)/sym(\mathfrak{h}^*)_{>0}^W$. The natural action of W on $sym(\mathfrak{h}^*)/sym(\mathfrak{h}^*)_{>0}^W$, coming from the action of W on the cartan, is the W action.
				\bs

				\textbf{The construction of the action for any $e\in \mathcal{N}$:} 
				(See \cite{L1}).
				Let $\pi_{\g}:\tilde{\g}\rightarrow \g$ be the Grothendieck resolution. One works in the setup of constructible and perverse sheaves. Let $\underline{\mathbb{Q}}$ be the constant sheaf. One first constructs an action of W on $R\pi_{\g,*}\underline{\mathbb{Q}}$:	Since $\pi_{\tilde{\g}}$ is small, $R\pi_{\g,*}\underline{\mathbb{Q}}[dim \g]$ is a perverse sheaf, moreover it is of the form $IC(\g^{rs},\mathcal{L})$, $\mathcal{L}:=((\pi_{\g,*})\underline{\mathbb{Q}})|_{\g^{rs}}$. (Pushforward of constant sheaf restricted to $\g^{rs}$. $\g^{rs}$ stands for regular semi simple elements, those over which $\tilde{\g}$ is an etale cover of rank $|W|$). Since $\tilde{\g^{rs}}\rightarrow \g^{rs}$ is an etale cover of rank $|W|$, W acts on $\mathcal{L}$. By functoriality of IC extension, W acts on $R\pi_{\g,*}\underline{\mathbb{Q}}$. Then the restriction to a fiber $e\in \mathcal{N}$ gives the action on $H^*(\mathcal{B}_e)$. 
				
	\section{The parameter space}

			To construct the local system for a symplectic resolution $X\rightarrow Y$ we need a base:
			Let $V_{\R}:=Pic(X)\otimes \R = H^2(X,\R)$.
			Let $H_i$ be the set of hyperplanes where derived localization doesn't hold.
			Let $V_{\R}^0$ to be the complement of that set of hyperplanes. Let $V_{\CC}^0:=V_{\R}^0\otimes\CC$ be the complexification. That's the topological space which is the base for our local system construction.
			
			Observe that for the springer resolution $T^*{G/B}\rightarrow \mathcal{N}$, $V_{\R}^0$ is the space of regular weights. Hence, $\pi_1(V_{\CC}^0)\simeq Br_{aff,pure}$ - affine pure braid group.
			
			\bs

			\subsubsection*{Another interpretation of the base space}
			
					Recall that for all symplectic resolutions $X^{(i)}$ of Y, the Picard groups are canonically identified. Lets denote this group Pic. One can take the ample cones $C(X^{(i)})_+$ of different symplectic resolutions of Y. There is a Coxeter group, $\mathbb{W}$, which is attached to Y by Namikawa \cite{Na} . In the case of $T^*G/P$ this is the ordinary Weyl group $W_L$. This group has an action on $Pic$. Consider all the ample cones of different symplectic resolutions of Y, and their transitions via the action of $\mathbb{W}$. Let $V_{\R}^{0'}$ be the complement of the boundaries of these cones in $Pic\otimes \R$. There is a conjecture that there is an isomorphism between $V_{\R}^{0'}\simeq V_{\R}^0$. 
			
			\bs
			
			There is a local system of categories on the parameter space $V^0_{\CC}$ with the value the category the category $D^b(\mathcal{D}(G/P)-mod_0)\simeq D^b(Coh_0(T^*G/P))$, which for the case P=B, generalizes the well known weak action of the affine Braid group on this category. Let $\G$ denote the groupoid of the space $V^0_{\CC}$. (Remark: By a local system we mean a weak local system. Attaching categories to points in $V^0_{\CC}$, attaching functors up to isomorphism to homotopy classes of paths between points, s.t concatenation of paths corresponds to composition of the functors. We do not discuss higher compatibilities).

			We first build a local system with value $D^b(\mathcal{D}(G/P)-mod)$. The restriction of the functors attached to paths, to the subcategory $D^b(\mathcal{D}(G/P)-mod_0)$ will give the required local system.
			
			In order to build it, it's convenient to use the presentation by generators and relations for the groupoid $\G$ given by Selvatti. 
			
	\section{The groupoid of the base space -generators and relations}

		\subsection*{Salvetti's generators and relations for groupoid of complexification of completion of real hyperplane arrangement}

			\subsubsection*{Salvetti idea} 
			Given a real hyperplane arrangement in $\mathbb{R}^n$. one looks at the complexification $W_{\CC}^0$ of the complement $W_{\R}^0$, and wants to find a presentation for the groupoid of that space. By constructing a CW complex embedded in $W_{\CC}^0$, whose embedding is a homotopy equivalence, Salvetti is able to describe the generators and relations for the groupoid in a combinatorial way. Generators are the 1 cells, relations are given by the boundaries(attaching maps) of the 2 cells. The zero cells in the CW complex, correspond to the real alcoves in $W^0_{\R}$.

						\subsubsection*{Generators}
						The generators of the groupoid are the positive half loops, going between alcoves A,A' that share a codim 1 face. Let $l_{A,A'}$ be the generator for the path from A to A' 
				
						\subsubsection*{Relations:}
						To express the relations we need to define a notion: 
								\emph{notion: length of path} Given a path in $V_{\CC}^0$ consisting of generators, its length is the number of generators involved. 
						
								\subsubsubsection{First set of relations}	
								one way to describe the relations is: for each two alcoves A,A': all positive minimal length orbit between A and A' are homotopic.
						
								\subsubsubsection{Smaller set of relations:}
						
						However, It's sufficient to take a smaller set of relations:
						
						for each codim 2 face F, and alcove A, that has F in its boundary, there are exactly two positive minimal paths between A and its opposite with respect to F, $A_F^-$. The relation we impose is that these paths are equivalent.

\section{Construction of the local system}

We specialize the setting of $X^{(i)}\rightarrow Y$ to the following case:
Fix a maximal torus $T\subset G$. Fix a levi $L\supset T$. Let P be any parabolic with the levi L. $X^{(i)}=$$T^*G/P^i$ $P^i$ with Levi L. We construct a local system with value $D^b(\mathcal{D}(G/P^i)-mod)$
		
		This category depends on P only up to the levi. 
		
	\bs

	In this section we discuss the parameter space $V^0_{\CC}$. Then we build a functor from the groupoid $\G(V^0_{\CC})$ to Cat, attaching a category to each alcove, and a functor between the categories of the alcoves for each path generator of the groupoid. To prove it's indeed a functor from the groupoid, we check the relations that guarantee that given two homotopy equivalent paths, the composition of the corresponding functors gives isomorphic functors.  
					\bs

	\subsection*{The parameter space, $V^0_{\CC}$, in the case $T^*G/P$}

						In this case, the parameter space $V^0_{\R}\otimes \CC\subset Pic(T^*G/P)\otimes \CC$, has an interpretation in terms of the root space of G.
						
						\begin{claim}
						
						$V_{\R}\simeq \Lambda_L\otimes\R$, $V_{\R}^0\simeq $regular parabolic weights
						\end{claim}
						\begin{proof}
								To see that $V_{\R}\simeq \Lambda_L\otimes \R$, recall that there are canonicals equivalence $Pic(G/P)\otimes \R \simeq Pic(T^*(G/P))\otimes \R$, and $Pic(G/P)\simeq \Lambda_L$.

								Moreover, the walls in $V_{\R}$ are defined as the hyperplanes where derived localization do not hold. It's known that derived localization do not hold  exactly for singular parabolic weights. (that is $\lambda\in \Lambda_L\otimes \R$ s.t $<\lambda+\rho,\alpha^{\vee}>=0$ mod p for some coroots $\alpha^{\vee}$ with $\g_{\alpha}$ not subset of Lie(L)). 
								\bs
						\end{proof}		
								For convenience denote $\mathfrak{h}_{\R}:=\Lambda\otimes \R$, $\mathfrak{h}_{\mathbb{R},L}:=\Lambda_L\otimes \R$, $\mathfrak{h}_{\R}:=\mathfrak{h}_{\R,L}^{reg}$ (=$V_{\R}^0$ in this case.)

						Denote the walls of singular weights, by $H_{\alpha,n}$
						
						$H_{\alpha,n}:=\lambda\in \Lambda_L\otimes \R |<\lambda+\rho,\alpha^{\vee}>=np$
						
						These affine root hyperplanes is the hyperplane arrangement in $V^0_{\R}$ in this case.
						
						\bs
						
						$\mathfrak{h}^0_{\R}\otimes \CC$ is the topological space which is the base for our local system construction. Observe that for P=B, a Borel, $\pi_1(V_{\CC}^0)\simeq Br_{aff,pure}$ - affine pure braid group.

				\subsection*{The categories attached to the alcoves}
				
				\subsubsection*{Key Lemma $\Gamma(\Dl(G/P))\simeq \Gamma(\Dl(G/Q))$} 

In order to construct a local system with value - the category $D^b(\mathcal{D}(G/P)-mod)$, we have to use different associated parabolic subgroups P,Q. The following lemma about the global sections of the sheaves of differential operators on these spaces will be key in that construction.   

		\begin{lemma}
	Let G be an algebraic groups over algebraically closed field, k, characteristic $p>>0$. Let P,Q be parabolic subgroups with same levi L. Let $\lambda\in \Lambda_L$(parabolic integral weight). Let $A_{\lambda}(G/P):=\Gamma(\mathcal{D}(G/P)_{\lambda})$ 
	
  Then, for a parabolic integral weight, $\lambda\in (\Lambda_L)$ There is an isormophism $\Gamma(\Dl(G/P))\simeq \Gamma(\Dl(G/Q))$. 
 
 \end{lemma}

\begin{proof}
There is a morphism $\ul\rightarrow \Al$. It is the morphism [[1]]$\ul=\Gamma(\Dl(G/B))\rightarrow \Gamma(\Dl(G/P))$ induced from the natural map $G/B\rightarrow G/P$

It is surjective for $char(k)>>0$ (\cite{BMR})
Denote its kernel by $ker_P$. It's enough to show $ker_P= ker_Q$ as subgroups of $u_{\lambda}$. The proof of this equality for $char(k)>>0$ reduces to the same claim in the setup where the base field is $\mathbb{C}$. 

\newcommand{\M}{\mathcal{M}}
\begin{claim}
Let $j:U\hookrightarrow G/P$ be affine open subset of G/P.
Let $\M'$ be a $D_{\lambda}$ module on U. Let $\M:=j_*\M'$. Let $M:=\Gamma(G/P,\M)$. consider it as $\ul$ module under [[1]]. Then $ker_P=Ann_{\ul}M$. 
\bs

\end{claim}

\begin{proof} (of claim)
Let $M':=\Gamma(U,\M')=M$
Consider the composition $\Al\hookrightarrow \Gamma(U,D_{\lambda})\hookrightarrow End(M')$. Both maps are injectives. the first by D affiness of G/P, the second since Weyl algebra is simple. Hence $ker_P\simeq Ann_{\ul}(M')$.

\end{proof}

Proof of Lemma 1:
By the claim, it's enough to find $D_{\lambda}$ modules, $\M_1,\M_2$ on G/P and on G/Q respectively that comes from D module on affine open, and have the same kernel under the action of $\ul$. 

The required modules $\M_1,\M_2$ are - on G/P, let [P]:=1P be the trivial coset. Let $\M_1:=\delta_{[P]}$. On G/Q Let [Q]:=1Q be the trivial coset. Let U:=P.[Q] be the orbit under the action of P. $j:U\hookrightarrow G/Q$ is affine open. Let $\M_2:=j_*O_{U}$. 

\end{proof}

							\subsubsection*{The categories attached to alcoves}
							
								For any parabolic Q with the Levi L, for each $\lambda\in \Lambda_L$ the sheaf $\mathcal{D}_{\lambda}(G/Q)$ is well defined. (using the canonical equivalence $Pic(G/Q)\simeq \Lambda_L$)

								By the 'key lemma', for two parabolic subgroups with the levi L, there is a canonical isomorphism of the global sections algbera $\Gamma(\mathcal{D}(G/P)_{\lambda})\simeq \Gamma(\mathcal{D}(G/Q)_{\lambda})$. Denote this algebra $A_{\lambda}$

								To  each real alcove $\A\subset V^0_{\R}$. attach the category
								$D^b(A_{\lambda}-mod)$ using some $\lambda\in \A\cap \Lambda_L;$  
			
								\bs
								
								\subsubsubsection{Independence of $\lambda$}
								The construction we give will be independent of the choice of $\lambda$ in an alcove A, in the sense that there is the canonical equivalences between the categories $D^b(A_{\lambda}-mod)\rightarrow D^b(A_{\lambda'}-mod)$ for $\lambda,\lambda'\in \mathcal{A}$, (given by localization, tensoring with $O(\lambda'-\lambda)$ and taking global sections). And the functors $F_{\lambda,\mu}$ that we construct for different paths, will sit in a square commutative diagram. 
								
								for $\lambda,\lambda'\in \A $ and $\mu,\mu'$ in adjucent alcoves.

\[
\xymatrix
{D^b(A_{\lambda}-mod)\ar[r]^-{F^{\lambda,\mu}}\ar[d]_{\simeq}
& D^b(A_{\mu}-mod)\ar[d]^-{\simeq} \\
D^b(A_{\lambda'}-mod)\ar[r]^-{F^{\lambda',\mu'}} &  D^b(A_{\mu'}-mod)}
\]

	\subsection*{Functors attached to the generators of the groupoid}
				We now construct the functors to be attached to paths. For the definition we need the following notion:		 
				
				\subsubsection*{Partial order on alcoves, defined by a cone}
		
									\begin{defi}
									Given a real hyperplane arrangement, fix a cone $A_0$. Define: two alcoves A,A' have relation "A' is above A wrt $A_0$" if $A'\subset A+A_0$. Notation is $A'>_{A_0}A$.

									\end{defi}
				
	Remark: an equivalent notion is: $A'>_{A_0}A\iff$ the positive half loop from A to A' belongs to $h_{\mathbb{R}}+i A_0$. 
						
				\subsubsection*{Definition of the functor $F_{\A,\A'}$, assigned to the generator $l_{\A,\A'}$}
				
				Given adjacent alcoves $\A,\A'$, there is a parabolic P, with the property that $\A<_P \A'$. Equivalently, $l_{\A,\A'}\subset V^0_{\R}+i\Lambda_P^+$ where $\Lambda_P^+$ is the positive weights cone for P in $\Lambda_L$.

				Let $\lambda_{\A}$, $\lambda_{\A'}$ be weights in these alcoves.
				Define the associated functor $F_{\A,\A'}^P:D^b(\Gamma(\D_{\lambda_{\A}}(G/P))-mod)\rightarrow D^b(\Gamma(\D_{\lambda_{\A'}}(G/P))-mod)$ to be 
				\begin{equation}
				F^P_{A,A'}:=\Gamma^{P,\lambda_A} (-\otimes O(\lambda_{A'}-\lambda_{A})) Loc^{P,\lambda_A'}.
				\end{equation}
				Where 
				$\Gamma^{P,\lambda_{\A}}:D^b(\mathcal{D}_{\lambda_{\A}}(G/P)-mod)\rightarrow D^b(\Gamma(\D_{\lambda_{\A}}(G/P))-mod), 
				Loc^{P,\lambda_{\A'}}:D^b(\Gamma(\D_{\lambda_{\A'}}(G/P))-mod)\rightarrow D^b(\mathcal{D}_{\lambda_{\A'}}(G/P)-mod)$, are global sections and localization functors respectively.  
				
				\begin{lemma}
				
				The functor (1) is independent of a choice of parabolic P, for which $\A<_PA'$. 
				
				\end{lemma}
				
				The proof of this lemma will be given in a later section. This claim is the reason, we denote this functor by a symbol $F_{A,A'}$, omitting P from the notation.

			\subsection*{Proof that the relations of the groupoid hold}
								
										To prove our construction is a functor from the groupoid to Cat, we need to check the relations hold. For this we use the following concrete claim.  
										
								\begin{claim}
										Consider our hyperplane arrangement $V_{\R}^0$. Let F be a codimension 2 face. Let $\A$ be an alcove which contains F in its boundary. Let $\Aop$ denote the opposite alcove with respect to F. Let C be the cone in $V_{\R}^0$, defined by the hyperplanes in the boundary of the alcove $\A$, that intersect the face F. Then the two minimal paths in $V_{\CC}^0$ from $\Aop$ to $\A$ are going up with respect to the cone C.   
								\end{claim}
									That's easy to verify.
										
								\begin{claim} Given an alcove $\A\subset V_{\R}^0$, and a codimension 2 face F in its boundary. There is a parabolic P(with the levi L), such that the 2 positive minimal orbits from $\Aop$ to $\A$ are orbits that always go up with respect to the partial order on alcoves given by the positive cone $\Lambda_P^+$ of P.  						
								\end{claim}
								\begin{proof}
								By claim 3, it works for P whose positive cone corresponds to a cone defined by the alcove $\A$ and face F. 
								\end{proof}

								\bs
								
								Finally, It will follow that the construction is indeed a functor from the groupoid if we prove the following claim
								
								\begin{claim}
								$F_{A,A'}$ satisfies the following relation: Any two paths between two alcoves, that go through increasing alcoves according to a fixed parabolic P, have isomorphic corresponding functors. 
						
								\end{claim}

						\begin{proof}
						Follows from definition since $\Gamma^{P,\lambda_{A'}}Loc^{P,\lambda_{A'}}\simeq Id$, and $(-\otimes O(\lambda'-\lambda))\otimes O(\lambda''-\lambda')\simeq (-\otimes O(\lambda''-\lambda))$
								
						\end{proof}

\subsection*{Proof of Lemma 2: $F_{\A,\A'}$ is well defined}

Recall - Let $\A,\A'$ be adjacent alcoves in $V_{\R}^0$ with shared codim 1 face. Let P and Q be two parabolic subgroups that have the Levi L, and satisfy $\A<_{P}\A'$ and $\A<_{Q}\A'$. Lemma 2 is the claim that there is an isomorphism of  functors $F_{\A,\A'}^P\simeq F_{\A,\A'}^Q$, where $F_{\A,\A'}^P:=\Gamma^{P,\lambda_A} (-\otimes O(\lambda_{A'}-\lambda_{A})) Loc^{P,\lambda_A'}.$
\bs

Proof Plan: Let B be a borel that contains T and contained in P. $T\subset B\subset P$. First, we explain the independence of choice of P, for the case of levi L s.t P=B. Then for general levi L, and parabolic P, we use this fact, and a commutative diagram that contains the functors $F_{\A,\A'}^P,F_{\A,\A'}^B$ and the * pullback functor pullback functor of D modules at the level of global sections $\pi^*:D^b(\Gamma(\mathcal{D}(G/P)_{\lambda})-mod)\rightarrow D^b(\Gamma(\mathcal{D}(G/B)_{\lambda})-mod)$ , ($\lambda\in (\Lambda_L)^{reg}$), to prove the isomorphism $F_{\A,\A'}^P\simeq F_{\A,\A'}^Q$. 

\bs

The following lemma assures that for the special case where L is such that the parabolic P is a Borel P=B, the functor $F^{B}_{\lambda,\mu}$ attached to weights $\lambda$, is independent of B. (remember the triples $\lambda,\mu,B$ that appear in $F^B_{\lambda,\mu}$ satisfy $\lambda<_B\mu$) 
	\begin{lemma} 
	The functor $F^{B}_{\lambda,\mu}:\Gamma(\mathcal{D}_{\lambda}(G/B))-mod\rightarrow \Gamma(\mathcal{D}_{\mu}(G/B))-mod$ is isomorphic to translation functor $T_{\lambda,\mu}:u(\g)_{\lambda}-mod\rightarrow u(\g)_{\mu}-mod$

	i.e

	\[
\xymatrix
{D^b (\Gamma(\mathcal{D}_{\lambda}(G/B))-mod)\ar[r]^-{F_{\lambda,\mu}^B} \ar[d]_{\simeq}& D^b(\Gamma(\mathcal{D}_{\mu}(G/B))-mod)  \ar[d]^-{\simeq}\\
D^b(u(\g)_{\lambda}-mod) \ar[r]^-{T_{\lambda,\mu}} & D^b(u(\g)_{\mu}-mod)}
\]

\end{lemma}
	
	This follows from \cite{BM}
	\bs

Next there is a commutative diagram:

Let $\pi:G/B\rightarrow G/P$ be the projection.
		\begin{lemma}
		
		Let $\lambda\in \A,\mu\in \A'$ be two parabolic regular integral weights (that is in $\Lambda_L^{reg}$), Recall that $\A<_P\A'$. 
		
	 The ordinary pullback functor $\pi^*:D^b \mathcal{D}_{\lambda} ((G/P))-mod\rightarrow D^b\mathcal{D}_{\lambda} ((G/B))-mod$, fits into the following commutative diagram that involves $F_{\mathcal{A},\mathcal{A}'}^P$: 
	\bs

\[
\xymatrix
{D^b(\Gamma(\mathcal{D}_{\lambda}(G/B))-mod) \ar[r]^-{F_{\lambda,\mu}^B} & D^b(\Gamma(\mathcal{D}_{\mu}(G/B))-mod) \\
D^b (\Gamma(\mathcal{D}_{\lambda}(G/P)-mod))\ar[r]^-{F_{\lambda,\mu}^P} \ar[u]^-{\pi*}& D^b(\Gamma(\mathcal{D}_{\mu}(G/P))-mod)  \ar[u]_{\pi*}}
\]
	\end{lemma}

\bs

	\begin{cor} 
	
	Using Lemma 3, lemma 4 becomes the following commutative diagram [[2]].

	\[
\xymatrix
{D^b(u(\g)_{\lambda}-mod) \ar[r]^-{T_{\lambda,\mu}} & D^b(u(\g)_{\mu}-mod) \\
D^b (A_{\lambda}-mod))\ar[r]^-{F_{\lambda,\mu}^P} \ar[u]^-{\pi*}& D^b(A_{\mu}-mod)  \ar[u]_{\pi*}}
\]
	
	where the vertical maps $\pi^*$ come from the surjection of algebras $u(\g)_{\lambda}\rightarrow A_{\lambda}$. It is derived tensoring with $u_{\lambda}\otimes_{A_{\lambda}}$ on the left verticle map, (similarly $u_{\lambda}\otimes_{A_{\lambda}}$ on the right verticle map)
	\end{cor}
	
	Observe also that since $u_{\lambda}\rightarrow A_{\lambda}$ is surjective, it follows that, 
	\begin{claim}
	The functor $A_{\lambda}-mod\rightarrow u(\g)_{\lambda}-mod$ is fully faithful and conservative in the abelian level
	\end{claim}

	\subsubsection*{Proof that the functor $F_{\lambda,\mu}^P:D^b(A_{\lambda}-mod)\rightarrow D^b(A_{\mu}-mod)$ is independent of the choice of P as long as $\lambda<_P\mu$}
	
		Even though by claim 6, the functor $\pi^*$ is fully faithful at the abelian level, this doesn't appriory imply that at the derived categories level - the upper horizontal functor $T_{\lambda,\mu}$ , determines the lower horizontal functor $F_{\lambda,\mu}^P$.  And yet 
		\begin{lemma}
		The functor $T_{\lambda,\mu}$ at the upper row of [[2]] does determine the functor at the lower row, $F_{\lambda,\mu}^P$.
		
		\end{lemma}
		
		\begin{proof}
			
		The two horizontal functors of the commutative diagram , $T_{\lambda,\mu}$ and $F_{\lambda,\mu}^P$, are given by derived tensoring with bimodules. The bimodules are the value of the functors on the algebras $F_{\lambda,\mu}^B(u_{\lambda})$,  $F_{\lambda,\mu}^P(A_{\lambda})$ respectively. By the definition of $F_{\lambda,\mu}^B, F_{\lambda,\mu}^P$, these are $R\Gamma(_{\lambda}\mathcal{D}_{\mu})$, where $\mathcal{D}_{\lambda}$ is the sheaf of differential operators either on G/B or G/P. and $_{\lambda}\mathcal{D}_{\mu}:=\mathcal{D}_{\lambda}\otimes O({\mu-\lambda})$.
		
		\begin{claim}
				These bimodules live in the heart of the modules categories $D^b(u_{\lambda}\otimes u_{\mu}^{op} -mod)$ $D^b(A_{\lambda}\otimes A_{\mu}^{op} -mod)$ respectively. 
		\end{claim}
		
		(It's enough to prove for P.)
		\begin{proof} 
				It's enough to prove that there is a filtration on $\mathcal{F}:=_{\lambda}\mathcal{D}_{\mu}$ , whose associated graded gr$\mathcal{F}$ has $H^i(gr\mathcal{F})=0$ for $i>0$. Since then it follows that $H^i(\mathcal{F})=0$ hence $R\Gamma(\mathcal{F})$ is in the heart. Indeed, $_{\lambda}\mathcal{D}_{\mu}$ has a filtration with associated graded $O_{T^*(G/P)}(\mu-\lambda)$, and for $\mu-\lambda>0$ $H^i(O_{T^*G/P}(\mu-\lambda))=0$. 
		
		\end{proof}
		
		Denote the bimodules, $B_{u}\in u_{\lambda}\otimes u_{\mu}^{op}-mod$, and $B_{A}\in A_{\lambda}\otimes A_{\mu}^{op}-mod$ respectively. 
		Let $B_A'\in A_{\lambda}\otimes u_{\mu}^{op}-mod$ be the module obtained from $B_A$ by making the $A_{\mu}$ right action to $u_{\mu}$ right action, through the map $u_{\mu}\rightarrow A_{\mu}$. Let $B_{u}':=A_{\lambda}\otimes_{u_{\lambda}}B_u$. $B_u'\in u_{\lambda}\otimes u_{\mu}^{op}-mod$. By the commutativity of diagram [[2]], it follows that $B_A'\simeq B_u'$. Hence $B_A'$ is determined from $B_u$. Since $A_{\mu}-mod\rightarrow u_{\mu}-mod$ is fully faithful and conservative, (similarly for $A_{\lambda}\otimes A_{\mu}^{op}-mod\rightarrow A_{\lambda}\otimes u_{\mu}^{op}-mod$) it follows that $B_A$ is determined from $B_A'$. Hence $B_A$ is determined from $B_u$.

			\end{proof}
	
		This finishes of the construction of the local system for $T^*G/P$

\end{document}